\documentclass[journal]{IEEEtran}
\usepackage{amsthm}
\newtheorem{theorem}{Theorem}
\newtheorem{lemma}{Lemma}
\theoremstyle{definition}
\newtheorem{definition}{Definition}

\usepackage{subfigure}
\usepackage{float}
\usepackage{amsfonts}
\usepackage{verbatim}
\usepackage{amssymb}

\usepackage{algorithm}
\usepackage{algorithmic}

\ifCLASSINFOpdf
\else
   \usepackage[dvips]{graphicx}
   \graphicspath{{../eps/}}
   \DeclareGraphicsExtensions{.eps}
\fi
\ifCLASSOPTIONcompsoc
 \usepackage[caption=false,font=normalsize,labelfont=sf,textfont=sf]{subfig}
\else
 \usepackage[caption=false,font=footnotesize]{subfig}
\fi
\begin{document}
%
\title{A New Nonconvex Strategy to Affine Matrix Rank Minimization Problem}
%
%
%

\author{Angang~Cui,
        Jigen~Peng,
        Haiyang~Li,
        Junxiong Jia,
        and~Meng~Wen
\thanks{A. Cui, J. Peng and J. Jia are with the School of Mathematics and Statistics, Xi'an Jiaotong University, Xi'an, 710049, China.
e-mail: (cuiangang@163.com; jgpengxjtu@126.com; jjx323@xjtu.edu.cn).}
\thanks{H. Li and M. Wen are with the School of Science, Xi'an Polytechnic University, Xi'an, 710048, China. e-mail: (fplihaiyang@126.com; wen5495688@163.com).}
\thanks{Manuscript received, ; revised , .}}

%
%

\markboth{Journal of \LaTeX\ Class Files,~Vol.~, No.~, ~}%
{Shell \MakeLowercase{\textit{et al.}}: Bare Demo of IEEEtran.cls for IEEE Journals}
%



\maketitle

\begin{abstract}
The affine matrix rank minimization (AMRM) problem is to find a matrix of minimum rank that satisfies a given linear system constraint. It has many applications in some
important areas such as control, recommender systems, matrix completion and network localization. However, the problem (AMRM) is NP-hard in
general due to the combinational nature of the matrix rank function. There are many alternative functions have been proposed to substitute the matrix rank function,
which lead to many corresponding alternative minimization problems solved efficiently by some popular convex or nonconvex optimization algorithms. In this paper, we
propose a new nonconvex function, namely, $TL_{\alpha}^{\epsilon}$ function (with $0\leq\alpha<1$ and $\epsilon>0$), to approximate the rank function, and translate the
NP-hard problem (AMRM) into the $TL_{p}^{\epsilon}$ function affine matrix rank minimization (TLAMRM) problem. Firstly, we study the equivalence of problem (AMRM) and (TLAMRM),
and proved that the uniqueness of global minimizer of the problem (TLAMRM) also solves the NP-hard problem (AMRM) if the linear map $\mathcal{A}$ satisfies a restricted
isometry property (RIP). Secondly, an iterative thresholding algorithm is proposed to solve the regularization problem (RTLAMRM) for all $0\leq\alpha<1$ and $\epsilon>0$. At last,
some numerical results on low-rank matrix completion problems illustrated that our algorithm is able to recover a low-rank matrix, and the extensive numerical on image
inpainting problems shown that our algorithm performs the best in finding a low-rank image compared with some state-of-art methods.
\end{abstract}

\begin{IEEEkeywords}
Affine matrix rank minimization problem, $TL_{\alpha}^{\epsilon}$ function, Equivalence, Iterative thresholding algorithm.
\end{IEEEkeywords}

%
\IEEEpeerreviewmaketitle

\section{Introduction}\label{section1}
The problem of recovering a low-rank matrix from a given linear system constraint, namely, affine matrix rank minimization (AMRM) problem, has been actively studied in
different fields such as control \cite{faze1,faze2}, recommender systems \cite{candes3,jan4}, matrix completion \cite{recht5,candes6,faze7,candes8,cui9} and network
localization \cite{ji11}. This rank minimization problem can be described as follows
\begin{equation}\label{equ1}
(\mathrm{AMRM})\ \ \ \ \ \ \min_{X\in \mathbb{R}^{m\times n}} \ \mbox{rank}(X)\ \ \mathrm{s.t.} \ \  \mathcal{A}(X)=b,
\end{equation}
where $b\in \mathbb{R}^{d}$ is a given vector, and $\mathcal{A}: \mathbb{R}^{m\times n}\mapsto \mathbb{R}^{d}$ is a linear map determined by $d$ matrices
$A_{1}, A_{2}, \cdots, A_{p}\in \mathbb{R}^{m\times n}$, i.e.,
$$\mathcal{A}(X):=\big(\langle A_{1},X\rangle, \langle A_{2},X\rangle,\cdots, \langle A_{d},X\rangle\big)^{\top}\in \mathbb{R}^{d}$$
with $\langle A_{i},X\rangle=\mathrm{trace}(A_{i}^{\top}X)$, $i=1,2,\cdots,d$. Without loss of generality, we assume that $m\leq n$ throughout this paper.
An important special case of the problem (AMRM) is the matrix completion (MC) problem:
\begin{equation}\label{equ2}
(\mathrm{MC})\ \ \ \ \ \ \min_{X\in \mathbb{R}^{m\times n}} \ \mbox{rank}(X)\ \ \mathrm{s.t.} \ \  X_{i,j}=M_{i,j},\ \ (i,j)\in \Omega,
\end{equation}
where $X, M\in \mathbb{R}^{m\times n}$ are both $m\times n$ real matrices, $\Omega$ is the set of indices of samples and the subset $\{M_{i,j}| (i,j)\in \Omega\}$ of the
entries is known. This problem has been widely applied in signal and image processing \cite{faze7,singer12}, machine learning \cite{srebro13}, computer vision \cite{hu14} and
the famous Netflix problem \cite{net15}. Unfortunately, problem (\ref{equ1}) is NP-hard \cite{recht5,faze7} for which all known finite time algorithms have at least doubly
exponential running times in both theory and practice. To overcome such a difficulty, Recht\cite{recht5}, Fazel \cite{faze7} and other researchers (e.g., \cite{candes3,candes6,cai16})
introduced the convex envelope of $\mathrm{rank}(X)$ on the set $\{X\in \mathbb{R}^{m\times n}: \|X\|_{2}\leq1\}$, namely, nuclear-norm $\|X\|_{\ast}$ of $X$, to relax the rank of $X$.
It leads to the nuclear-norm affine matrix rank minimization (NAMRM) problem
\begin{equation}\label{equ3}
(\mathrm{NAMRM})\ \ \ \ \ \ \min_{X\in \mathbb{R}^{m\times n}} \ \|X\|_{\ast}\ \ \mathrm{s.t.} \ \  \mathcal{A}(X)=b
\end{equation}
for the constrained problem and
\begin{equation}\label{equ4}
(\mathrm{RNAMRM})\ \ \ \ \ \ \min_{X\in \mathbb{R}^{m\times n}} \Big\{\|\mathcal{A}(X)-b\|_{2}^{2}+\lambda\|X\|_{\ast}\Big\}
\end{equation}
for the regularized unconstrained problem, where $\lambda>0$ is the regularization parameter and $\|X\|_{\ast}=\sum_{i=1}^{m}\sigma_{i}(X)$ is defined as the sum of the nonzero singular
values of $X\in \mathbb{R}^{m\times n}$.

Recht et al.\cite{recht5} have shown that if a certain restricted isometry property holds for the linear map $\mathcal{A}$, the minimum rank solution can be recovered by solving
the problem (NAMRM), and the sharp results can be seen in \cite{cai17,tony18}. Many algorithms for solving the problems (NAMRM) and (RNAMRM) have been
proposed. These include semidefinite programming and interior point SDP solver \cite{recht5,tut19}, singular value thresholding (SVT) algorithm \cite{cai16}, accelerated proximal
gradient (APG) algorithm \cite{tho20}, inexact proximal point algorithms \cite{liu21}, fixed point and Bregman iterative algorithms \cite{ma22,gol23}. However, the problem (RNAMRM)
may yield a matrix with much higher rank and need more observations to recover a real low-rank matrix \cite{candes3}, and it may tend to lead to biased estimation by shrinking all
the singular values toward zero simultaneously \cite{cai16}.

On the other hand, with recent development of non-convex relaxation approaches in sparse signal recovery problems, a
large number of non-convex surrogate functions have been proposed to approximate the $\l_{0}$-norm, including $\l_{p}$-norm $(0<p<1)$ \cite{char24,xu25,fou26,lai27,chen28,dau29,mou30,sun31,char32,peng33,xu34,cao35}, MCP (Mini-max Concave Plus) \cite{zhang36}, SCAD (Smoothly Clipped Absolute Deviation) \cite{fan37},
Laplace \cite{west38,Trzasko39}, Logarithm \cite{thi40}, capped $\l_{1}$-norm \cite{zhang41}, smoothed $\l_{0}$-norm \cite{moh42}. Inspired by the good performance of the non-convex
surrogate functions in sparse signal recovery problems, these popular nonconvex surrogate functions have been extended on the singular values to better approximate the rank function
(e.g., \cite{lu43,lu44,li45,zhang46,chen47,Moh48,lai49,moh50}). Some empirical evidence has also shown that the corresponding non-convex algorithms can really make a better recovery
in some matrix rank minimization problems. Different from previous studies, in this paper, a new continuous promoting low-rank function
\begin{equation}\label{equ5}
TL_{\alpha}^{\epsilon}(X)=\sum_{i=1}^{m}\varphi_{\alpha}^{\epsilon}(\sigma_{i}(X))=\sum_{i=1}^{m}\frac{(\sigma_{i}(X))^{1/2}}{(\sigma_{i}(X)+\epsilon)^{1/2-\alpha}}
\end{equation}
in terms of the singular values of matrix $X$ is considered to approximate the rank function, where the continuous function
\begin{equation}\label{equ6}
\varphi_{\alpha}^{\epsilon}(|t|)=\frac{|t|^{1/2}}{(|t|+\epsilon)^{1/2-\alpha}}
\end{equation}
is the $TL_{\alpha}^{\epsilon}$ function for all $0\leq\alpha<1$ and $\epsilon>0$ . It is easy to verify that the $TL_{\alpha}^{\epsilon}$  function $\varphi_{\alpha}^{\epsilon}$ is concave for
any $\alpha\in (0,1/2]$. Moreover, with the change of parameters $\alpha$ and $\epsilon$, we have
$$
\lim_{\alpha\rightarrow 0^{+}}\lim_{\epsilon\rightarrow0^{+}}\varphi_{\alpha}^{\epsilon}(|t|)=\left\{
    \begin{array}{ll}
      0, & {\mathrm{if} \ t=0;} \\
      1, & {\mathrm{if} \ t\neq 0,}
    \end{array}
  \right.
$$
and therefore the function (\ref{equ5}) interpolates the rank of matrix $X$:
\begin{equation}\label{equ7}
\begin{array}{llll}
&&\displaystyle\lim_{\alpha\rightarrow 0^{+}}\lim_{\epsilon\rightarrow0^{+}}TL_{\alpha}^{\epsilon}(X)\\
&&=\displaystyle\lim_{\alpha\rightarrow 0^{+}}\lim_{\epsilon\rightarrow0^{+}}\sum_{i=1}^{m}\varphi_{\alpha}^{\epsilon}(\sigma_{i}(X))\\
&&=\displaystyle\lim_{\alpha\rightarrow 0^{+}}\lim_{\epsilon\rightarrow0^{+}}\sum_{i=1}^{m}\frac{(\sigma_{i}(X))^{1/2}}{(\sigma_{i}(X)+\epsilon)^{1/2-\alpha}}\\
&&=\mathrm{rank}(X).
\end{array}
\end{equation}
Then, by this transformation, we propose the new approximation optimization problem of the problem (AMRM) which has the following form
\begin{equation}\label{equ8}
(\mathrm{TLAMRM})\ \ \ \min_{X\in \mathbb{R}^{m\times n}} TL_{\alpha}^{\epsilon}(X)\ \ \mathrm{s.t.} \ \  \mathcal{A}(X)=b
\end{equation}
for the constrained problem and
\begin{equation}\label{equ9}
(\mathrm{RTLAMRM})\ \ \ \min_{X\in \mathbb{R}^{m\times n}} \Big\{\|\mathcal{A}(X)-b\|_{2}^{2}+\lambda TL_{\alpha}^{\epsilon}(X)\Big\}
\end{equation}
for the regularization problem.

This paper is organized as follows. Section \ref{section2} presents some useful notions and crucial preliminary results that are used in this paper. Section \ref{section3}
presents the equivalence between minimization problems (AMRM) and (TLAMRM). Section \ref{section4} presents an iterative thresholding algorithm to solve the problem (RTLAMRM)
for all $0\leq\alpha<1$ and $\epsilon>0$. The experimental results are presented in Section \ref{section5}. Finally, some conclusion remarks are presented in Section \ref{section6}.

\section{Notions and preliminary results} \label{section2}
In this section, we present some useful notions and crucial preliminary results that are used in this paper.

\subsection{Notions} \label{section2-1}
The space of $m\times n$ real matrices is denoted by $\mathbb{R}^{m\times n}$. Given any $X\in\mathbb{R}^{m\times n}$, the Frobenius norm of $X$ is
denoted by $\|X\|_{F}$, namely, $\|X\|_{F}=\sqrt{\mathrm{tr}(X^{\top}X)}$, where $\mathrm{tr}(\cdot)$ denotes the trace of a matrix. Given any matrices
$X,Y\in \mathbb{R}^{m\times n}$, the standard inner product of matrices $X$ and $Y$ is denoted by $\langle X,Y\rangle$, and $\langle X,Y\rangle=\mathrm{Tr}(Y^{\top}X)$.
The linear map $\mathcal{A}:\mathbb{R}^{m\times n}\mapsto \mathbb{R}^{d}$ determined by $d$ matrices $A_{1}, A_{2}, \cdots, A_{d}\in \mathbb{R}^{m\times n}$
is given by $\mathcal{A}(X)=(\langle A_{1},X\rangle, \langle A_{2},X\rangle,\cdots, \langle A_{d},X\rangle)^{\top}\in \mathbb{R}^{d}$. Define
$A=(vec(A_{1}), vec(A_{2}), \cdots, vec(A_{d}))^{\top}\in \mathbb{R}^{d\times mn}$ and $x=vec(X)\in \mathbb{R}^{mn}$, we have $\mathcal{A}(X)=Ax$.
Let $\mathcal{A}^{\ast}$ denote the adjoint of $\mathcal{A}$. Then for any $y\in \mathbb{R}^{d}$, we have $\mathcal{A}^{\ast}(y)=\sum_{i=1}^{d}y_{i}A_{i}$.
The singular value decomposition of matrix $X\in \mathbb{R}^{m\times n}$ is $X=U_{X}[\mathrm{Diag}(\sigma(X)),\mathbf{0}_{m,n-m}] V_{X}^{\top}$, where $U_{X}$ is an $m\times m$
unitary matrix, $V_{Y}$ is an $n\times n$ unitary matrix, $[\mathrm{Diag}(\sigma(X)),\mathbf{0}]\in \mathbb{R}^{m\times n}$, $\mathbf{0}_{m,n-m}\in \mathbb{R}^{m,n-m}$ is a $m\times (n-m)$
zero matrix, and the vector $\sigma(X): \sigma_{1}(X)\geq\sigma_{2}(X)\geq\cdots\geq\sigma_{r}(X)\geq\sigma_{r+1}(X)=\cdots=\sigma_{m}(X)=0$, arranged in
descending order, denotes the singular value vector of matrix $X$.

\subsection{Preliminary results} \label{section2-2}

\begin{lemma}\label{lem1}{\rm(see \cite{recht5})}
Let $M, N\in \mathbb{R}^{m\times n}$. Then there exist matrices $N_{1}, N_{2}\in \mathbb{R}^{m\times n}$ such
that\\
(1) $N=N_{1}+N_{2}$;\\
(2) $\mathrm{rank}(N_{1})\leq 2\mathrm{rank}(M)$;\\
(3) $MN_{2}^{\top}=\mathbf{0}_{m,m}$ and $M^{\top}N_{2}=\mathbf{0}_{n,n}$;\\
(4) $\langle N1_{1}, N_{2}\rangle=0$.
\end{lemma}

\begin{lemma}\label{lem2}
Let $M, N\in \mathbb{R}^{m\times n}$. If $MN^{\top}=\mathbf{0}_{m, m}$ and $M^{\top}N=\mathbf{0}_{n, n}$, then
\begin{equation}\label{equ10}
TL_{\alpha}^{\epsilon}(M+N)=TL_{\alpha}^{\epsilon}(M)+TL_{\alpha}^{\epsilon}(N).
\end{equation}
\end{lemma}

\begin{proof}
Consider the singular value decompositions of matrices $M$ and $N$:
$$M=U_{M}[\mathrm{Diag}(\sigma(M)),\mathbf{0}_{m,n-m}]V_{M}^{\top},$$
$$N=U_{N}[\mathrm{Diag}(\sigma(N)),\mathbf{0}_{m,n-m}]V_{N}^{\top}.$$
Since the unitary matrices $U_{M}, U_{N}\in \mathbb{R}^{m\times m}$ are invertible, the condition $MN^{\top}=\mathbf{0}_{m, n}$ implies that $V_{M}^{\top}V_{N}=\mathbf{0}_{n, n}$.
Similarly, $M^{\top}N=\mathbf{0}_{n, n}$ implies that $U_{M }^{\top}U_{N}=\mathbf{0}_{m, m}$. Thus, the following is a valid SVD for $M+N$,
\begin{eqnarray*}
&&M+N\\
&&=\left[
      \begin{array}{cc}
        U_{M} & U_{N} \\
      \end{array}
    \right]\cdot\\
&&\left[
  \begin{array}{cccc}
    \mathrm{Diag}(\sigma(M)) & \mathbf{0}_{m,n-m} & \mathbf{0}_{m,m} & \mathbf{0}_{m,n-m} \\
    \mathbf{0}_{m, m} & \mathbf{0}_{m,n-m} & \mathrm{Diag}(\sigma(N)) & \mathbf{0}_{m,n-m} \\
  \end{array}
\right]\cdot\\
&&\left[
  \begin{array}{cc}
    V_{M} & V_{N} \\
  \end{array}
\right]^{\top},
\end{eqnarray*}
which implies that the singular values of $M+N$ are equal to the union (with repetition) of the singular values of $M$ and $N$.
Hence, we get the equation (\ref{equ10}).
\end{proof}

\begin{lemma}\label{lem3}
Let $X=U_{X}\mathrm{Diag}(\sigma(X))V_{X}^{\top}$ be the singular value decomposition of matrix $X\in \mathbb{R}^{m\times n}$, and $\mathrm{rank}(X)=r$. For any $\alpha\in[0,1/2]$ and
$\epsilon\in(0,1/3]$, then there exists
\begin{equation}\label{equ11}
\eta_{1}=3r\sigma_{1}(X)
\end{equation}
such that, for any $\eta\geq \eta_{1}$,
\begin{equation}\label{equ12}
\sum_{i=1}^{m}\frac{\sigma_{i}(\eta^{-1}X)}{(\sigma_{i}(\eta^{-1}X)+\epsilon)^{1-2\alpha}}\leq \frac{1}{3\epsilon^{1-2\alpha}}.
\end{equation}
\end{lemma}

\begin{proof}
Since the function $t/(t+\epsilon)^{1-2\alpha}$ is increasing in $t\in[0,+\infty)$, we have
\begin{equation}\label{equ13}
\begin{array}{llll}
&&\displaystyle\sum_{i=1}^{m}\frac{\sigma_{i}(\eta^{-1}X)}{(\sigma_{i}(\eta^{-1}X)+\epsilon)^{1-2\alpha}}\\
&&\leq\displaystyle\frac{r\sigma_{1}(\eta^{-1}X)}{(\sigma_{1}(\eta^{-1}X)+\epsilon)^{1-2\alpha}}\\
&&\leq\displaystyle \frac{r\sigma_{1}(X)}{\eta\epsilon^{1-2\alpha}}.
\end{array}
\end{equation}
In order to get equation (\ref{equ12}), it suffices to impose
\begin{equation}\label{equ14}
\frac{r\sigma_{1}(X)}{\eta\epsilon^{1-2\alpha}}\leq \frac{1}{3\epsilon^{1-2\alpha}},
\end{equation}
equivalently,
$$\eta \geq 3r\sigma_{1}(X).$$
This completes the proof.
\end{proof}

\begin{lemma}\label{lem4}
(\cite{xu25}) For any fixed $\lambda>0$ and $y_{i}\in \mathbb{R}$, let
\begin{equation}\label{equ15}
h_{\lambda}(y_{i}):=\arg\min_{x_{i}\geq 0}\Big\{(x_{i}-y_{i})^{2}+\lambda x_{i}^{1/2}\Big\},
\end{equation}
then the half thresholding function $h_{\lambda}$ can be analytically expressed by
\begin{equation}\label{equ16}
h_{\lambda}(y_{i})=\left\{
    \begin{array}{ll}
      h_{\lambda,1/2}(y_{i}), & \ \mathrm{if} \ {y_{i}> \frac{\sqrt[3]{54}}{4}\lambda^{2/3};} \\
      0, & \ \mathrm{if} \ {y_{i}\leq\frac{\sqrt[3]{54}}{4}\lambda^{2/3},}
    \end{array}
  \right.
\end{equation}
where
\begin{equation}\label{equ17}
h_{\lambda,1/2}(y_{i})=\frac{2}{3}t\Big(1+\cos\Big(\frac{2\pi}{3}-\frac{2}{3}\phi_{\lambda}(y_{i})\Big)\Big)
\end{equation}
with
\begin{equation}\label{equ18}
\phi_{\lambda}(y_{i})=\arccos\Big(\frac{\lambda}{8}\Big(\frac{|y_{i}|}{3}\Big)^{-3/2}\Big).
\end{equation}
\end{lemma}

\begin{definition}\label{def1}
(\cite{xu25}) For any $\lambda>0$ and $y=(y_{1}, y_{2}, \cdots, y_{m})^{\top}\in \mathbb{R}^{m}$, the vector half thresholding
operator $H_{\lambda}$ is defined as
\begin{equation}\label{equ19}
H_{\lambda}(y)=(h_{\lambda}(y_{1}), h_{\lambda}(y_{2}), \cdots, h_{\lambda}(y_{m}))^{\top},
\end{equation}
where $h_{\lambda}$ is defined in Lemma \ref{lem3}.
\end{definition}

\begin{lemma}\label{lem5}
(\cite{xu25}) For any $y_{i}>\frac{\sqrt[3]{54}}{4}\lambda^{2/3}$, the half thresholding function $h_{\lambda}(y_{i})$ defined in (\ref{equ15}) is strict increasing.
\end{lemma}

\begin{definition}\label{def2}
Suppose matrix $Y\in \mathbb{R}^{m\times n}$ admits a singular value decomposition as $Y=U_{Y}[\mathrm{Diag}(\sigma(Y)),\mathbf{0}_{m,n-m}]V_{Y}^{\top}$.
For any $\lambda>0$, the matrix half thresholding operator $\mathcal{H}_{\lambda}: \mathbb{R}^{m\times n}\rightarrow \mathbb{R}^{m\times n}$ is defined by
\begin{equation}\label{equ20}
\mathcal{H}_{\lambda}(Y)=U_{Y}[\mathrm{Diag}(H_{\lambda}(\sigma(Y))),\mathbf{0}_{m,n-m}]V_{Y}^{\top},
\end{equation}
where $H_{\lambda}$ is defined in Definition \ref{def1}.
\end{definition}

The  matrix half thresholding operator $\mathcal{H}_{\lambda}$ simply applies the vector half thresholding operator $H_{\lambda}$ defined in Definition \ref{def1} to the singular
value vector of a matrix, and effectively shrinks the singular values towards zero. If there are some nonzero singular values of matrix $Y$ are below the threshold value $\frac{\sqrt[3]{54}}{4}\lambda^{2/3}$, we can immediately get that the rank of $\mathcal{H}_{\lambda}(Y)$ lower than the rank of matrix $Y$.

Combing Lemma \ref{lem5} and Definition \ref{def2}, we can get the following crucial Lemma.
\begin{lemma}\label{lem6}
Let $Y=U_{Y}[\mathrm{Diag}(\sigma_{i}(Y)),\mathbf{0}_{m,n-m}]V_{Y}^{\top}$ be the singular value decomposition of matrix $Y\in \mathbb{R}^{m\times n}$ and $\mathcal{H}_{\lambda}(Y)=U_{Y}[\mathrm{Diag}(H_{\lambda}(\sigma(Y))),\mathbf{0}_{m,n-m}]V_{Y}^{\top}$.
Then
\begin{equation}\label{equ21}
\mathcal{H}_{\lambda}(Y)=\arg\min_{X\in \mathbb{R}^{m\times n}}\Big\{\|X-Y\|_{F}^{2}+\lambda \|X\|^{1/2}_{1/2}\Big\}.
\end{equation}
\end{lemma}

\begin{proof}
Similar argument as used in the proof of (\cite{yu51}, Theorem 2.1).
\end{proof}

\begin{definition}\label{def3}
Let $Y=U_{Y}[\mathrm{Diag}(\sigma_{i}(Y)),\mathbf{0}_{m,n-m}]V_{Y}^{\top}$ be the singular value decomposition of matrix $Y$, and $\sigma(Y): \sigma_{1}(Y)\geq\sigma_{2}(Y)\geq\cdots\geq\sigma_{r}(Y)\geq\sigma_{r+1}(Y)=\cdots=\sigma_{m}(Z)=0$  be the singular value vector of matrix $Y\in \mathbb{R}^{m\times n}$,
for any $\alpha\in(0,1)$, we define
\begin{equation}\label{equ22}
\begin{array}{llll}
&&\mathcal{H}_{\lambda/(\sigma(Z)+\epsilon)^{1/2-\alpha}}(Y)\\
&&=U_{Y}[\mathrm{Diag}(H_{\lambda/(\sigma(Z)+\epsilon)^{1/2-\alpha}}(\sigma(Y))),\mathbf{0}_{n,n-m}]V_{Y}^{\top},
\end{array}
\end{equation}
where
\begin{equation}\label{equ23}
\begin{array}{llll}
&&H_{\lambda/(\sigma(Z)+\epsilon)^{1/2-\alpha}}(\sigma(Y))\\
&&=(h_{\lambda/(\sigma_{1}(Z)+\epsilon)^{1/2-\alpha}}(\sigma_{1}(Y)), h_{\lambda/(\sigma_{2}(Z)+\epsilon)^{1/2-\alpha}}(\sigma_{2}(Y)),\\
&&\ \ \ \ \cdots, h_{\lambda/(\sigma_{m}(Z)+\epsilon)^{1/2-\alpha}}(\sigma_{m}(Y)))^{\top}
\end{array}
\end{equation}
is defined in Definition \ref{def1}, and $h_{\lambda/(\sigma_{i}(Z)+\epsilon)^{1/2-\alpha}}$ is obtained by replacing $\lambda$ with $\lambda/(\sigma_{i}(Y)+\epsilon)^{1/2-\alpha}$ in $h_{\lambda}$.
\end{definition}


\section{Equivalence between minimization problems (AMRM) and (TLAMRM)}\label{section3}

In this section, we shall establish the equivalence between minimization problems (AMRM) and (TLAMRM). We demonstrate that the minimizer of the problem (TLAMRM)
also solves the problem (AMRM) if some specific conditions are satisfied.

\begin{definition}\label{de4} {\rm(\cite{gol23})}
For every integer $r$ with $1\leq r\leq m$, the linear operator $\mathcal{A}: \mathbb{R}^{m\times n}\mapsto \mathbb{R}^{d}$ is said to satisfy the Restricted Isometry Property (RIP)
with the restricted isometry constant $\delta_{r}(\mathcal{A})$ if $\delta_{r}(\mathcal{A})$ is the minimum constant that satisfies
\begin{equation}\label{equ24}
(1-\delta_{r}(\mathcal{A}))\|X\|_{F}^{2}\leq\|\mathcal{A}(X)\|_{2}^{2}\leq(1+\delta_{r}(\mathcal{A}))\|X\|_{F}^{2}
\end{equation}
for all $X\in \mathbb{R}^{m\times n}$ with $\mathrm{rank}(X)\leq r$, and $\delta_{r}(\mathcal{A})$ is called the RIP constant. Note that $\delta_{s}(\mathcal{A})\leq \delta_{t}(\mathcal{A})$,
if $s\leq t$.
\end{definition}

The RIP concept and the RIP constant $\delta_{r}(\mathcal{A})$ play a key role in the relationship between the NP-hard original problem (AMRM) and its convex relaxation problem (NAMRM).
Especially, Recht et al. \cite{recht5} have proved that the problem (AMRM) and (NAMRM) have the same optimal solution if the RIP constant satisfies $\delta_{5r}(\mathcal{A})<0.1$.
A very natural idea appears in our mind: can we also have the same solution to the problems (AMRM) and (TLAMRM)? In the following theorem, we will give answer to this question.

\begin{theorem}\label{the1}
Let $X^{\ast}$ and $X_{0}$ be the minimizers to the problem (TLAMRM) and (AMRM) respectively. For any $\alpha\in[0,1/2]$ and $\epsilon\in(0,1/3]$, if there is a number $k>2\beta$, such that
\begin{equation}\label{equ25}
\epsilon^{1-2\alpha}(2\beta)^{-3/2}\sqrt{1-\delta_{2\beta+k}(\mathcal{A})}-\sqrt{\frac{1+\delta_{k}(\mathcal{A})}{k}}>0,
\end{equation}
then the unique minimizer $X^{\ast}$ of the problem (TLAMRM) is exactly $X_{0}$, where $\beta=\mathrm{rank}(X_{0})$.
\end{theorem}

\begin{proof}
Let $E=X^{\ast}-X_{0}$. Applying Lemma \ref{lem1} to the matrices $X_{0}$ and $E$, there exist matrices $E_{0}$ and $E_{c}$ such that
$E=E_{0}+E_{c}$, $\mathrm{rank}(E_{0})\leq 2\mathrm{rank}(X_{0})$, $X_{0}E_{c}^{\top}=\mathbf{0}_{m,m}$, $X_{0}^{\top}E_{c}=\mathbf{0}_{n.n}$ and $\langle E_{0},E_{c}\rangle=0$. Then
\begin{equation}\label{equ26}
\begin{array}{llll}
TL_{\alpha}^{\epsilon}(X_{0})&\geq&
TL_{\alpha}^{\epsilon}(X^{\ast})\\
&=& TL_{\alpha}^{\epsilon}(X_{0}+E)\\
&\geq& TL_{\alpha}^{\epsilon}(X_{0}+E_{c})-TL_{\alpha}^{\epsilon}(E_{0})\\
&=& TL_{\alpha}^{\epsilon}(X_{0})+TL_{\alpha}^{\epsilon}(E_{c})\\
&&-TL_{\alpha}^{\epsilon}(E_{0}),
\end{array}
\end{equation}
where the first inequality follows from the optimality of $X^{\ast}$, third assertion follows the triangle inequality and the last one follows Lemma \ref{lem2}. Rearranging (\ref{equ26}), we can
conclude that
\begin{equation}\label{equ27}
TL_{\alpha}^{\epsilon}(E_{0})\geq TL_{\alpha}^{\epsilon}(E_{c}).
\end{equation}

We partition $E_{c}$ into a sum of matrices $E_{1}, E_{2}, \cdots$, each of rank at most $k$. Let $E_{c}=U_{E_{c}}[\mathrm{Diag}(\sigma(E_{c})),\mathbf{0}]V_{E_{c}}^{\top}$
be the singular value decomposition of matrix $E_{c}$. For each $j\geq1$, define the index set $I_{j}=\{k(j-1)+1, \cdots, kj\}$, and let
$E_{j}=U_{E_{I_{j}}}[\mathrm{Diag}(\sigma(E_{j})),\mathbf{0}]V_{E_{I_{j}}}^{\top}$ (notice that $\langle E_{k},E_{l}\rangle=0$ if $k\neq l$). For each $\nu\in I_{j}$, by
Lemma \ref{lem3}, there exist $\gamma_{1}=3k\sigma_{1}(E_{c})$, for any $\gamma>\gamma_{1}$, $\alpha\in[0,1/2]$ and $\epsilon\in(0,1/3]$, we have
\begin{eqnarray*}
\frac{\sigma_{\nu}(\gamma^{-1}E_{j})}{(\sigma_{\nu}(\gamma^{-1}E_{j})+\epsilon)^{1-2\alpha}}&\leq&
\sum_{\nu\in I_{j}}\frac{\sigma_{\nu}(\gamma^{-1}E_{j})}{(\sigma_{\nu}(\gamma^{-1}E_{j})+\epsilon)^{1-2\alpha}}\\
&\leq& \frac{1}{3\epsilon^{1-2\alpha}}.
\end{eqnarray*}
Also since
$$\frac{\sigma_{\nu}(\gamma^{-1}E_{j})}{(\sigma_{\nu}(\gamma^{-1}E_{j})+\epsilon)^{1-2\alpha}}\leq \frac{1}{3\epsilon^{1-2\alpha}}\Leftrightarrow \sigma_{\nu}(\gamma^{-1}E_{j})\leq \xi$$
where $\xi$ is an unknown positive constant, and particularly we can choose $\xi=1/3$ for any $\alpha\in[0,1/2]$ and $\epsilon\in(0,1/3]$. Thus, we have
$$\sigma_{\nu}(\gamma^{-1}E_{j})\leq\frac{\sigma_{\nu}(\gamma^{-1}E_{j})}{(\sigma_{\nu}(\gamma^{-1}E_{j})+\epsilon)^{1-2\alpha}},\ \ \ \ \forall \nu\in I_{j}.$$
Moreover, by the construction of matrices $E_{i}$s, we can get that
\begin{equation}\label{equ28}
\sigma_{\tilde{\mu}}(\gamma^{-1}E_{j+1})\leq\frac{\sum_{v\in I_{j}}\sigma_{\nu}(\gamma^{-1}E_{j})}{k},\ \ \ \ \forall \tilde{\mu}\in I_{j+1}.
\end{equation}
It follows that
\begin{equation}\label{equ29}
\|\gamma^{-1}E_{j+1}\|_{F}\leq \frac{1}{\sqrt{k}}\sum_{\nu\in I_{j}}\frac{\sigma_{\nu}(\gamma^{-1}E_{j})}{(\sigma_{\nu}(\gamma^{-1}E_{j})+\epsilon)^{1-2\alpha}}
\end{equation}
and
\begin{equation}\label{equ30}
\begin{array}{llll}
&&\displaystyle\sum_{j\geq2}\|\gamma^{-1}E_{j+1}\|_{F}\\
&&\leq\displaystyle\sum_{j\geq1}\bigg(\frac{1}{\sqrt{k}}
\sum_{\nu\in I_{j}}\frac{(\sigma_{\nu}\gamma^{-1}E_{j})}{(\sigma_{\nu}(\gamma^{-1}E_{j})+\epsilon)^{1-2\alpha}}\bigg)\\
&&\leq\displaystyle\frac{1}{\sqrt{k}}\sum_{i=1}^{m}\frac{\sigma_{i}(\gamma^{-1}E_{0})}{(\sigma_{i}(\gamma^{-1}E_{0})+\epsilon)^{1-2\alpha}}\\
&&\leq\displaystyle\frac{1}{\sqrt{k}}\bigg(\sum_{i=1}^{m}\frac{(\sigma_{i}(\gamma^{-1}E_{0}))^{1/2}}{(\sigma_{i}(\gamma^{-1}E_{0})+\epsilon)^{1/2-\alpha}}\bigg)^{2}\\
&&=\displaystyle\frac{1}{\sqrt{k}}(TL_{\alpha}^{\epsilon}(\gamma^{-1}E_{0}))^{2}
\end{array}
\end{equation}
where the second inequality follows from (\ref{equ27}).

At the next step, we will derive two inequalities between the Frobenius norm and function $TL_{\alpha}^{\epsilon}$. Since
$$\frac{(\sigma_{i}(\gamma^{-1}E_{0}))^{1/2}}{(\sigma_{i}(\gamma^{-1}E_{0})+\epsilon)^{1/2-\alpha}}\leq \frac{(\sigma_{i}(\gamma^{-1}E_{0}))^{1/2}}{\epsilon^{1/2-\alpha}},$$
we have
\begin{equation}\label{equ31}
\begin{array}{llll}
TL_{\alpha}^{\epsilon}(\gamma^{-1}E_{0})&=&\displaystyle\sum_{i=1}^{m}\frac{(\sigma_{i}(\gamma^{-1}E_{0}))^{1/2}}{(\sigma_{i}(\gamma^{-1}E_{0})+\epsilon)^{1/2-\alpha}}\\
&\leq&\displaystyle\sum_{i=1}^{m}\frac{(\sigma_{i}(\gamma^{-1}E_{0}))^{1/2}}{\epsilon^{1/2-\alpha}}\\
&\leq& \|\gamma^{-1}E_{0}\|_{F}^{1/2}(2\beta)^{3/4}/\epsilon^{1/2-\alpha}\\
&\leq&\|\gamma^{-1}(E_{0}+E_{1})\|_{F}^{1/2}(2\beta)^{3/4}/\epsilon^{1/2-\alpha},
\end{array}
\end{equation}
where the second inequality follows from the H\"{o}lder's inequality.

Finally, we put all these together as

\begin{equation}\label{equ32}
\begin{array}{llll}
&&\|\mathcal{A}(\gamma^{-1}E)\|_{2}\\
&&=\|\mathcal{A}(\gamma^{-1}(E_{0}+E_{c}))\|_{2}\\
&&=\displaystyle\|\mathcal{A}(\gamma^{-1}(E_{0}+E_{1}))+\sum_{j\geq2}\mathcal{A}(\gamma^{-1}E_{j})\|_{2}\\
&&\geq\displaystyle\|\mathcal{A}(\gamma^{-1}(E_{0}+E_{1}))\|_{2}-\|\sum_{j\geq2}\mathcal{A}(\gamma^{-1}E_{j})\|_{2}\\
&&\geq\displaystyle\|\mathcal{A}(\gamma^{-1}(E_{0}+E_{1}))\|_{2}-\sum_{j\geq2}\|\mathcal{A}(\gamma^{-1}E_{j})\|_{2}\\
&&\geq\displaystyle\sqrt{1-\delta_{2\beta+k}(\mathcal{A})}\|\gamma^{-1}(E_{0}+E_{1})\|_{F}\\
&&\ \ \ \ -\displaystyle\sqrt{1+\delta_{k}(\mathcal{A})}\sum_{j\geq2}\|\gamma^{-1}E_{j}\|_{F}\\
&&=\displaystyle\epsilon^{1-2\alpha}(2\beta)^{-3/2}\sqrt{1-\delta_{2\beta+k}(\mathcal{A})}(TL_{\alpha}^{\epsilon}(\gamma^{-1}E_{0}))^{2}\\
&&\ \ \ \ -\displaystyle\sqrt{\frac{1+\delta_{k}(\mathcal{A})}{k}}(TL_{\alpha}^{\epsilon}(\gamma^{-1}E_{0}))^{2}.
\end{array}
\end{equation}

Since $\mathcal{A}(E)=\mathcal{A}(X^{\ast}-X_{0})=\mathbf{0}$, and the the factor
$$\epsilon^{1-2\alpha}(2\beta)^{-3/2}\sqrt{1-\delta_{2\beta+k}(\mathcal{A})}-\sqrt{\frac{1+\delta_{k}(\mathcal{A})}{k}}$$
is strictly positive, we can get that
$E_{0}=\mathbf{0}$.
Furthermore, according to (\ref{equ27}), we have $E_{c}=\mathbf{0}$. Therefore, $X^{\ast}=X_{0}$.
\end{proof}

\section{Iterative thresholding algorithm for solving the problem (RTLAMRM)}\label{section4}

In this section, we propose an $TL_{\alpha}^{\epsilon}$ iterative half thresholding (TLIHT) algorithm to solve the problem (RTLAMRM) for all $0\leq\alpha<1$ and $\epsilon>0$.

For any $\lambda,\mu, \epsilon\in(0,+\infty)$, $\alpha\in[0,1)$ and $Z\in \mathbb{R}^{m\times n}$, let
\begin{equation}\label{equ33}
C_{\lambda}(X)=\|\mathcal{A}(X)-b\|_{2}^{2}+\lambda \sum_{i=1}^{m}\frac{(\sigma_{i}(X))^{1/2}}{(\sigma_{i}(X)+\epsilon)^{1/2-\alpha}},
\end{equation}
\begin{equation}\label{equ34}
\begin{array}{llll}
C_{\lambda,\mu}(X, Z)&=&\mu\|\mathcal{A}(X)-b\|_{2}^{2}+\lambda\mu \displaystyle\sum_{i=1}^{m}\frac{(\sigma_{i}(X))^{1/2}}{(\sigma_{i}(Z)+\epsilon)^{1/2-\alpha}}\\
&&-\mu\|\mathcal{A}(X)-\mathcal{A}(Z)\|_{2}^{2}+\|X-Z\|_{F}^{2}
\end{array}
\end{equation}
and
\begin{equation}\label{equ35}
B_{\mu}(X)=X+\mu \mathcal{A}^{\ast}(b-\mathcal{A}(X)).
\end{equation}
Note that by (\ref{equ33}) and (\ref{equ34}), we have
\begin{equation}\label{equ36}
C_{\lambda,\mu}(X,X)=\mu C_{\lambda}(X).
\end{equation}

\begin{lemma}\label{lem7}
For any fixed $\lambda>0$, $\mu>0$, $\epsilon>0$, $\alpha\in[0,1)$ and $Z\in \mathbb{R}^{m\times n}$, if $X^{s}\in \mathbb{R}^{m\times n}$ is a global minimizer of $C_{\lambda,\mu}(X,Z)$, then
\begin{equation}\label{equ37}
X^{s}=\mathcal{H}_{\lambda\mu/(\sigma(Z)+\epsilon)^{1/2-\alpha}}(B_{\mu}(Z)),
\end{equation}
where $\mathcal{H}_{\lambda\mu/(\sigma(Z)+\epsilon)^{1/2-\alpha}}$ is obtained by replacing $\lambda/(\sigma(Z)+\epsilon)^{1/2-\alpha}$ with $\lambda\mu/(\sigma(Z)+\epsilon)^{1/2-\alpha}$ in $\mathcal{H}_{\lambda/(\sigma(Z)+\epsilon)^{1/2-\alpha}}$.
\end{lemma}

\begin{proof}
By definition, $C_{\mu}(X,Z)$ can be rewritten as
\begin{eqnarray*}
&&C_{\lambda,\mu}(X,Z)\\
&&=\|X-(Z-\mu \mathcal{A}^{\ast}\mathcal{A}(Z)+\mu \mathcal{A}^{\ast}(b))\|_{F}^{2}\\
&&\ \ \ +\lambda\mu \displaystyle\sum_{i=1}^{m}\frac{(\sigma_{i}(X))^{1/2}}{(\sigma_{i}(Z)+\epsilon)^{1/2-\alpha}}+\mu\|b\|_{2}^{2}\\
&&\ \ \ -\|Z-\mu \mathcal{A}^{\ast}\mathcal{A}(Z)+\mu \mathcal{A}^{\ast}(b)\|_{F}^{2}+\|Z\|_{F}^{2}\\
&&\ \ \ -\mu\|\mathcal{A}(Z)\|_{2}^{2}\\
&&=\|X-B_{\mu}(Z)\|_{F}^{2}+\lambda\mu \displaystyle\sum_{i=1}^{m}\frac{(\sigma_{i}(X))^{1/2}}{(\sigma_{i}(Z)+\epsilon)^{1/2-\alpha}}\\
&&\ \ \ +\mu\|b\|_{2}^{2}-\|B_{\mu}(Z)\|_{F}^{2}+\|Z\|_{F}^{2}-\mu\|\mathcal{A}(Z)\|_{2}^{2},
\end{eqnarray*}
which implies that minimizing $C_{\lambda,\mu}(X,Z)$ for any fixed  $\lambda>0$, $\mu>0$, $\epsilon>0$, $\alpha\in[0,1)$ and matrix $Z\in \mathbb{R}^{m\times n}$ is equivalent to
$$\min_{X\in \mathbb{R}^{m\times n}}\Big\{\|X-B_{\mu}(Z)\|_{F}^{2}+\lambda\mu \displaystyle\sum_{i=1}^{m}\frac{(\sigma_{i}(X))^{1/2}}{(\sigma_{i}(Z)+\epsilon)^{1/2-\alpha}}\Big\}.$$
By Lemma \ref{lem6}, the expression (\ref{equ37}) immediately follows.
\end{proof}

\begin{lemma}\label{lem8}
For any fixed $\lambda>0$, $\epsilon>0$ and $0<\mu<\frac{1}{\|\mathcal{A}\|_{2}^{2}}$. If $X^{\ast}$ is a global minimizer of $C_{\lambda}(X)$, then $X^{\ast}$ is also a global minimizer of
$C_{\lambda,\mu}(X,X^{\ast})$, that is
\begin{equation}\label{equ38}
C_{\lambda,\mu}(X^{\ast},X^{\ast})\leq C_{\lambda,\mu}(X,X^{\ast})
\end{equation}
for all $X\in \mathbb{R}^{m\times n}$.
\end{lemma}

\begin{proof}
The condition $0<\mu<\frac{1}{\|\mathcal{A}\|_{2}^{2}}$ implies that
\begin{eqnarray*}
&&\|X-X^{\ast}\|_{F}^{2}-\mu\|\mathcal{A}(X)-\mathcal{A}(X^{\ast})\|_{2}^{2}\\
&&\geq(1-\mu\|\mathcal{A}\|_{2}^{2})\|X-X^{\ast}\|_{F}^{2}\\
&&\geq0.
\end{eqnarray*}
Therefore, for any $X\in \mathbb{R}^{m\times n}$, we have
\begin{eqnarray*}
&&C_{\lambda,\mu}(X,X^{\ast})\\
&&=\mu\|\mathcal{A}(X)-b\|_{2}^{2}+\lambda\mu \sum_{i=1}^{m}\frac{(\sigma_{i}(X))^{1/2}}{(\sigma_{i}(X^{\ast})+\epsilon)^{1/2-\alpha}}\\
&&\ \ \ -\mu\|\mathcal{A}(X)-\mathcal{A}(X^{\ast})\|_{2}^{2}+\|X-X^{\ast}\|_{F}^{2}\\
&&\geq\mu\Big[\|\mathcal{A}(X)-b\|_{2}^{2}+\lambda \sum_{i=1}^{m}\frac{(\sigma_{i}(X))^{1/2}}{(\sigma_{i}(X^{\ast})+\epsilon)^{1/2-\alpha}}\Big]\\
&&\geq\mu C_{\lambda}(X^{\ast})\\
&&=C_{\lambda,\mu}(X^{\ast},X^{\ast}).
\end{eqnarray*}
This completes the proof.
\end{proof}

Lemma \ref{lem8} show us that, if $X^{\ast}$ is a global minimizer of $C_{\lambda}(X)$, it is also a global minimizer of $C_{\lambda,\mu}(X,Z)$ with $Z=X^{\ast}$.
Combing with Lemma \ref{lem6}, we now derive the following important alternative theorem, which underlies the algorithm to be proposed.

\begin{theorem}\label{the2}
For any fixed $\lambda>0$ and $0<\mu<\frac{1}{\|\mathcal{A}\|_{2}^{2}}$. Let $X^{\ast}\in \mathbb{R}^{m\times n}$ be a global solution of the problem (RTLAMRM) and
$B_{\mu}(X^{\ast})=X^{\ast}+\mu \mathcal{A}^{\ast}(b-\mathcal{A}(X^{\ast}))$ admit the following singular value decomposition
\begin{equation}\label{equ39}
B_{\mu}(X^{\ast})=U^{\ast}[\mathrm{Diag}(\sigma_{i}(B_{\mu}(X^{\ast}))),\mathbf{0}_{m,n-m}](V^{\ast})^{\top}.
\end{equation}
Then $X^{\ast}$ satisfies the following fixed point inclusion
\begin{equation}\label{equ40}
X^{\ast}=\mathcal{H}_{\lambda\mu/(\sigma(X^{\ast})+\epsilon)^{1/2-\alpha}}(B_{\mu}(X^{\ast}))\\
\end{equation}
where $\mathcal{H}_{\lambda\mu/(\sigma(X^{\ast})+\epsilon)^{1/2-\alpha}}$ is obtained by replacing $\lambda$ with $\lambda\mu/(\sigma(X^{\ast})+\epsilon)^{1/2-\alpha}$ in $\mathcal{H}_{\lambda}$,
which means that, per iteration, the singular values of matrix $X^{k+1}$ satisfy
\begin{equation}\label{equ41}
\begin{array}{llll}
&&\sigma_{i}(X^{k+1})=\\
&&\left\{
    \begin{array}{ll}
      h_{\lambda/(\sigma_{i}(X^{k})+\epsilon)^{1/2-\alpha}}(\sigma_{i}(B_{\mu}(X^{k}))), & \mathrm{if} \ {\sigma_{i}(B_{\mu}(X^{k}))> t^{\ast};} \\
      0, & \mathrm{if} \ {\sigma_{i}(B_{\mu}(X^{k}))\leq t^{\ast}.}
    \end{array}
  \right.
\end{array}
\end{equation}
for $i=1,\cdots,m$, where the threshold function $t^{\ast}$ is defined as
\begin{equation}\label{equ42}
t^{\ast}=\frac{\sqrt[3]{54}}{4}(\lambda\mu/(\sigma_{i}(X^{k})+\epsilon)^{1/2-\alpha})^{2/3}.
\end{equation}
\end{theorem}

With the representation (\ref{equ40}), the TLIHT algorithm for solving the problem (RTLAMRM) can be naturally given by
\begin{equation}\label{equ43}
X^{k+1}=\mathcal{H}_{\lambda\mu/(\sigma(X^{k})+\epsilon)^{1/2-\alpha}}(X^{k}+\mu \mathcal{A}^{\ast}(b-\mathcal{A}(X^{k})))
\end{equation}
for all $0\leq\alpha<1$ and $\epsilon>0$.

Next, we analyze the convergence of the above TLIHT algorithm, and the convergence of the TLIHT algorithm is very important in guaranteeing that the algorithm can be successfully applied.

\begin{theorem} \label{the3}
Given $\lambda>0$, let $\{X^{k}\}$ be the sequence generated by the TLIHT algorithm with the step size $\mu$ satisfying $0<\mu<\frac{1}{\|\mathcal{A}\|_{2}^{2}}$, then
\begin{description}
\item[$\mathrm{i)}$] The sequence $\{X^{k}\}$ is a minimization sequence, and the sequence $\{C_{\lambda}(X^{k})\}$ is decreasing and converges to $C_{\lambda}(X^{\ast})$,
where $X^{\ast}$ is any accumulation point of the sequence $\{X^{k}\}$.
\item[$\mathrm{ii)}$] The sequence $\{X^{k}\}$ is asymptotically regular, i.e., $\lim_{k\rightarrow\infty}\|X^{k+1}-X^{k}\|_{F}^{2}=0$.
\item[$\mathrm{iii)}$] Any accumulation point of the sequence $\{X^{k}\}$ is a stationary point of the problem (RTLAMRM).
\end{description}
\end{theorem}

\begin{proof}
Its proof follows from the fact that the step size $\mu$ satisfying $0<\mu<\frac{1}{\|\mathcal{A}\|_{2}^{2}}$, and a similar argument as used in the proof of (\cite{xu25}, Theorem 3).
\end{proof}

As we all know, the quality of the solution to a regularization problem depends seriously on the setting of the regularization parameter $\lambda>0$. However, the selection of proper
parameter is a very hard problem and there is no optimal rule in general. In this paper, we suppose that the matrix $X^{\ast}$ of rank $r$ is the optimal solution of the regularization
problem (RTLAMRM), and set
\begin{equation}\label{equ44}
\lambda=\frac{\sqrt{96}(\sigma_{r+1}(B_{\mu}(X^{k}))^{3/2}(\sigma_{r+1}(X^{k})+\epsilon)^{1/2-p}}{9\mu}
\end{equation}
in each iteration. That is, (\ref{equ44}) can be used to adjust the value of
the regularization parameter $\lambda$ during iteration, and the TLIHT algorithm will be adaptive and free from the choice of regularization parameter $\lambda$. Moreover, we also find
that the the quantity of the solution of the MIHT algorithm also depends seriously on the setting of the parameter $\epsilon$. In TLIHT algorithm, a proper choice for the value of
$\epsilon$ at $k$-th iteration is given by
\begin{equation}\label{equ45}
\epsilon=\max\{\sigma_{r+1}(X^{k}), 10^{-3}\}.
\end{equation}

\begin{algorithm}[h!]
\caption{: TLIHT algorithm}
\label{alg:A}
\begin{algorithmic}
\STATE {\textbf{Input}: $\mathcal{A}: \mathbb{R}^{m\times n}\mapsto \mathbb{R}^{d}$, $b\in \mathbb{R}^{d}$}
\STATE {\textbf{Initialize}: $X^{0}\in \mathbb{R}^{m\times n}$, $\mu=\frac{1-\eta}{\|\mathcal{A}\|_{2}^{2}}(\eta\in(0,1))$, $\alpha\in[0,1)$ and $\epsilon>0$;}
\STATE {$k=0$;}
\STATE {\textbf{while} not converged \textbf{do}}
\STATE {$B_{\mu}(X^{k})=X^{k}-\mu \mathcal{A}^{\ast}\mathcal{A}(X^{k})+\mu \mathcal{A}^{\ast}(b)$;}
\STATE {Compute the SVD of $B_{\mu}(X^{k})$ as}
\STATE {$B_{\mu}(X^{k})=U^{k}[\mathrm{Diag}(\sigma_{i}(B_{\mu}(X^{k}))),\mathbf{0}_{m,n-m}](V^{k})^{\top}$;}
\STATE {$\lambda=\frac{\sqrt{96}(\sigma_{r+1}(B_{\mu}(z^{k})))^{3/2}(\lceil z^{k}\rfloor_{r+1}+\epsilon)^{1/2-\alpha}}{9\mu}$;}
\STATE {$t^{\ast}=\frac{\sqrt[3]{54}}{4}(\lambda\mu/(\sigma_{i}(B_{\mu}(z^{k}))+\epsilon)^{1/2-\alpha})^{2/3}$;}
\STATE \ \ \ \ {\textbf{for}\ $i=1:m$}
\STATE \ \ \ \ {1.\ $\sigma_{i}(B_{\mu}(X^{k}))>t^{\ast}$, then}
\STATE \ \ \ \ {$\sigma_{i}(B_{\mu}(X^{k+1}))=h_{\lambda\mu/(\sigma_{i}(X^{k})+\epsilon)^{1/2-\alpha}}(\sigma_{i}(B_{\mu}(X^{k})))$;}
\STATE \ \ \ \ {2.\ $\sigma_{i}(B_{\mu}(X^{k}))<t^{\ast}$, then}
\STATE \ \ \ \ {$\sigma_{i}(B_{\mu}(X^{k+1}))=0$;}
\STATE \ \ \ \ {\textbf{end}}
\STATE {$X^{k+1}=U^{k}[\mathrm{Diag}(\sigma_{i}(B_{\mu}(X^{k+1}))),\mathbf{0}_{m,n-m}](V^{k})^{T}$;}
\STATE {$k\rightarrow k+1$;}
\STATE {\textbf{end while}}
\STATE {\textbf{return}: $X^{k+1}$;}
\end{algorithmic}
\end{algorithm}

\section{Numerical experiments}\label{section5}
In this section, we present a series of numerical experiments to test the performance of the TLIHT algorithm for some matrix completion problems and compare it
with some state-of-art methods (SVT algorithm \cite{cai16} and SVP algorithm \cite{meka52}) in some grayscale image inpainting problems.

\subsection{Completion of random matrices}
In this subsection, we present a series of numerical experiments to test the performance of the TLIHT algorithm for some random low rank matrix completion problems. In our
experiments, we aim to recover a random matrix $X\in \mathbb{R}^{m\times n}$ with rank $r$ from a subset of observe entries, $\{M_{i,j}| (i,j)\in \Omega\}$. We generate
random matrices $M_{1}\in \mathbb{R}^{n\times r}$ and $M_{2}\in \mathbb{R}^{r\times n}$ with independent identically distributed Gaussian entries. Let $M=M_{1}M_{2}$,
and the matrix $M$ has rank at most $r$. We sample the observe set $\Omega$ with the sampling ratio $sr=s/mn$, where $s$ is the cardinality of the set $\Omega$. One quantity
helps to quantify the difficulty of a recovery problem is the freedom ratio $fr=s/r(m+n-r)$, which is the freedom of rank $r$ matrix divided by the number of measurement.
If $fr<1$, there is always an infinite number of matrices with rank $r$ with the given entries, so we cannot hope to recover the matrix in this situation \cite{ma22}. The
stopping criterion used in our algorithm is defined as follows:
$$\frac{\|X^{k}-X^{k-1}\|_{F}}{\|X^{k}\|_{F}}\leq \mathrm{Tol},$$
where $\mathrm{Tol}$ is a given small number. In our numerical experiments, we set $\mathrm{Tol}=10^{-8}$. Given an approximate recovery $X^{\ast}$ for $M$, the relative
error is defined as
$$\mathrm{RE}=\frac{\|X^{\ast}-M\|_{F}}{\|M\|_{F}}.$$
In order to implement our algorithm, we need to determine the parameter $\alpha$, which influences the behaviour of penalty function  $TL_{\alpha}^{\epsilon}(X)$. In the numerical
tests, we test our algorithm on a series of low rank matrix completion problems with different $\alpha$ values, and set $\alpha=0,0.1,0.3,0.4,0.5,0.6,0.7,0.9$, respectively. We
only take $m=n=100$, and the results are shown in Tables \ref{table1}, \ref{table2}. Comparing the performances of TLIHT algorithm for matrix completion problems with different
rank $r$, parameter $\alpha$ and $fr$, we can find that the parameters $\alpha=0, 0.1, 0.3, 0.4, 0.5$ seem to be the optimal strategy for our algorithm when $fr$ is closed to one.

\begin{table*}\footnotesize
\caption{\scriptsize Numerical results of TLIHT algorithm for matrix completion problems with different rank $r$, parameter $\alpha$ and $fr$ but fixed $n$, $sr=0.40$.}\label{table1}
\centering
\setlength{\tabcolsep}{3.5mm}{
\begin{tabular}{|c||l|l|l|l|l|l|l|l|}\hline
Problem&\multicolumn{2}{c}{$\alpha=0$}&\multicolumn{2}{|c}{$\alpha=0.1$}&\multicolumn{2}{|c}{$\alpha=0.3$}&\multicolumn{2}{|c|}{$\alpha=0.4$}\\
\hline
($n$,\,$r$,\,$fr$)&RE&Time&RE&Time&RE&Time&RE&Time\\
\hline
$(200,\,22,\,1.9240)$&1.73e-07& 1.70& 1.56e-07& 1.43& 1.69e-07& 1.85& 1.65e-07& 1.73\\
\hline
$(200,\,24,\,1.7730)$&1.95e-07& 1.67& 2.36e-07& 1.77& 1.91e-07& 2.04& 1.82e-07& 1.89\\
\hline
$(200,\,26,\,1.6454)$&2.17e-07& 1.96& 2.77e-07& 2.25& 2.32e-07& 2.18& 2.23e-07& 2.05\\
\hline
$(200,\,28,\,1.5361)$&2.80e-07& 2.45& 3.49e-07& 2.44& 3.32e-07& 2.57& 3.19e-07& 2.69\\
\hline
$(200,\,30,\,1.4414)$&4.15e-07& 3.08&3.47e-07& 3.04& 4.04e-07& 3.32&  4.05e-07& 3.22\\
\hline
$(200,\,32,\,1.3587)$&5.48e-07& 4.06& 4.84e-07& 3.91& 4.76e-07& 3.95& 5.08e-07& 4.03\\
\hline
$(200,\,34,\,1.2858)$&7.30e-07& 5.67& 8.67e-07& 5.98& 7.49e-07& 5.62& 6.70e-07& 5.21\\
\hline
$(200,\,36,\,1.2210)$&1.05e-06& 7.78& 1.04e-06& 7.39& 1.10e-06& 7.39& 1.02e-06& 7.59\\
\hline
$(100,\,38,\,1.1631)$&1.74e-06& 12.01& 1.62e-06& 11.20& 2.00e-06& 12.67& 1.69e-06& 11.66\\
\hline
$(200,\,40,\,1.1111)$&3.22e-06& 21.00& 3.22e-06& 20.73& 3.44e-06& 22.05& 3.08e-06& 20.27\\
\hline
$(200,\,42,\,1.0641)$&9.11e-06& 49.18& 9.19e-06& 50.56& 8.89e-06& 50.27& 8.97e-06& 50.83\\
\hline
$(200,\,43,\,1.0423)$&1.90e-05& 90.36& 1.97e-05& 95.82& 1.97e-05& 92.70& 1.88e-05& 96.10\\
\hline
$(200,\,44,\,1.0215)$&6.77e-05& 275.82& 6.92e-05& 267.57& 7.21e-05& 296.01& 6.97e-05& 271.25\\
\hline
$(200,\,45,\,1.0016)$&--& --&  --& --& --& --& --& --\\
\hline
\end{tabular}}
\end{table*}

\begin{table*}\footnotesize
\caption{\scriptsize Numerical results of TLIHT algorithm for matrix completion problems with different rank $r$,  parameter $\alpha$ and $fr$ but fixed $n$, $sr=0.40$.}\label{table2}
\centering
\setlength{\tabcolsep}{3.5mm}{
\begin{tabular}{|c||l|l|l|l|l|l|l|l|}\hline
Problem&\multicolumn{2}{c}{$\alpha=0.5$}&\multicolumn{2}{|c}{$\alpha=0.6$}&\multicolumn{2}{|c}{$\alpha=0.7$}&\multicolumn{2}{|c|}{$\alpha=0.9$}\\
\hline
($n$,\,$r$,\,FR)&RE&Time&RE&Time&RE&Time&RE&Time\\
\hline
$(200,\,22,\,1.9240)$&1.50e-07& 1.82& 5.88e-07& 3.60&  3.19e-07& 2.70&  4.05e-07& 3.89\\
\hline
$(200,\,24,\,1.7730)$&2.33e-07& 2.00& 5.64e-07& 3.31& 4.98e-07& 4.26& 5.30e-07& 4.51\\
\hline
$(200,\,26,\,1.6454)$&3.07e-07& 2.60& 7.49e-07& 3.95& 4.41e-07& 3.48& 6.89e-07& 4.57\\
\hline
$(200,\,28,\,1.5361)$&3.34e-07& 3.32& 6.62e-07& 4.30& 8.59e-07& 5.15& 7.38e-07& 5.84\\
\hline
$(200,\,30,\,1.4414)$&3.89e-07& 3.97& 8.04e-07& 5.40& 6.42e-07& 5.67& 7.74e-07& 6.16\\
\hline
$(200,\,32,\,1.3587)$&6.52e-07& 4.93& 9.18e-07& 8.24& 1.10e-06& 6.29& 1.09e-06& 8.32\\
\hline
$(200,\,34,\,1.2858)$&7.83e-07& 5.60& 1.83e-06& 9.17& 1.58e-06& 9.42& 2.13e-06& 8.89\\
\hline
$(200,\,36,\,1.2210)$&1.10e-06& 7.94& 2.19e-06& 11.52& 2.46e-06& 12.85& 1.86e-06& 12.99\\
\hline
$(100,\,38,\,1.1631)$&2.10e-06& 12.93& 1.83e-06& 12.79& 1.83e-06& 12.92& 1.97e-06& 13.46\\
\hline
$(200,\,40,\,1.1111)$&3.15e-06& 20.52& 4.49e-06& 24.95& 6.28e-06& 26.68& 8.19e-06& 36.35\\
\hline
$(200,\,42,\,1.0641)$&8.82e-06& 49.65& 9.57e-06& 54.70& 1.00e-05& 56.40& 1.15e-05& 60.02\\
\hline
$(200,\,43,\,1.0423)$&1.88e-05& 92.21& 1.89e-05& 94.03& 2.09e-05& 104.64& 2.35e-05& 110.39\\
\hline
$(200,\,44,\,1.0215)$&7.44e-05& 285.03&7.70e-05& 314.93& 7.18e-05& 304.40& 8.43e-05&306.12\\
\hline
$(200,\,45,\,1.0016)$&--& --&  --& --& --& --& --& --\\
\hline
\end{tabular}}
\end{table*}

\subsection{Application for image inpainting}

In this subsection, we demonstrate the performance of the TLIHT algorithm on some image inpainting problems. The TLIHT algorithm is tested on two standard $256\times 256$ grace images
(Peppers and Cameraman). We first use the singular value decomposition to obtain their approximated images with rank $r=30$. Original images and their corresponding approximated images
are displayed in Figs \ref{fig1} and \ref{fig2}. We take $sr=0.40$ and $sr=0.30$ for the two low rank images. We only take $\alpha=0.1,0.5$ in the TLIHT algorithm. Numerical results
of the three algorithms for image inpainting problems are reported in Table \ref{table3}. We display the recovered Peppers and Cameraman images via the three algorithms in Figs \ref{fig3},
\ref{fig4} respectively. We can see that the TLIHT algorithm with $\alpha=0.1$ performs the best in image inpainting problems compared with SVT algorithm and SVP algorithm.

\begin{figure}[h!]
  \centering
  \begin{minipage}[t]{0.45\linewidth}
  \centering
  \includegraphics[width=1.1\textwidth]{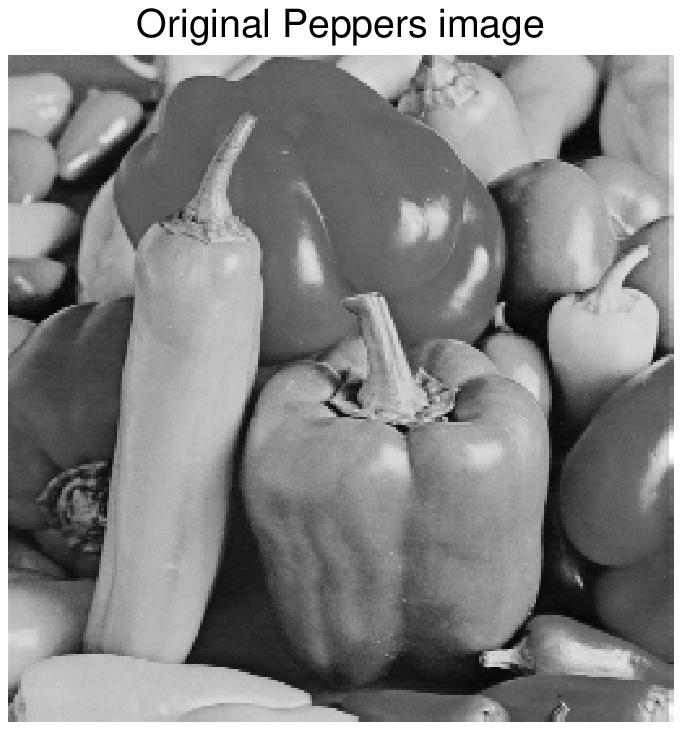}
  \end{minipage}
  \begin{minipage}[t]{0.45\linewidth}
  \centering
  \includegraphics[width=1.1\textwidth]{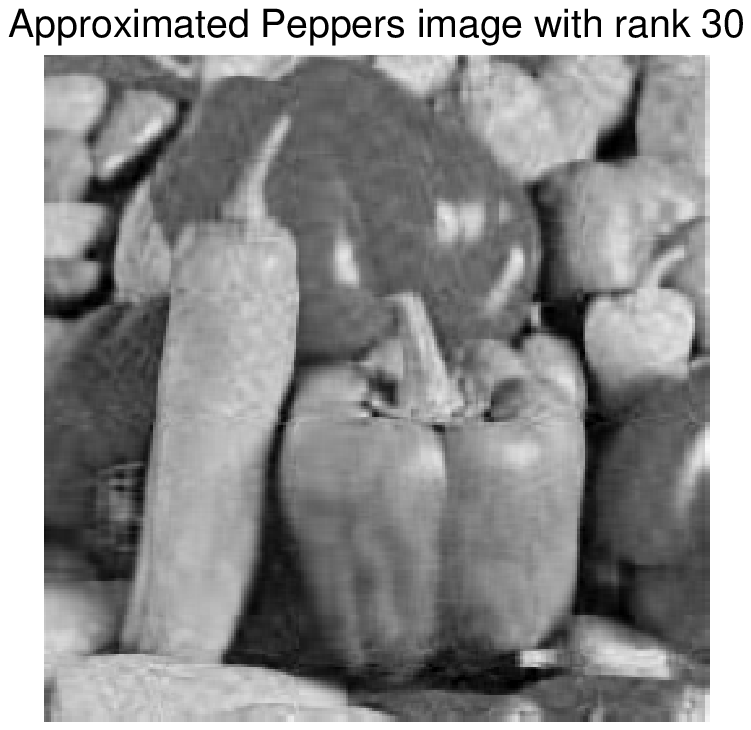}
  \end{minipage}
  \caption{Original $256\times 256$ Peppers image and its approximated image with rank 30} \label{fig1}
\end{figure}

\begin{figure}[h!]
  \centering
  \begin{minipage}[t]{0.45\linewidth}
  \centering
  \includegraphics[width=1.1\textwidth]{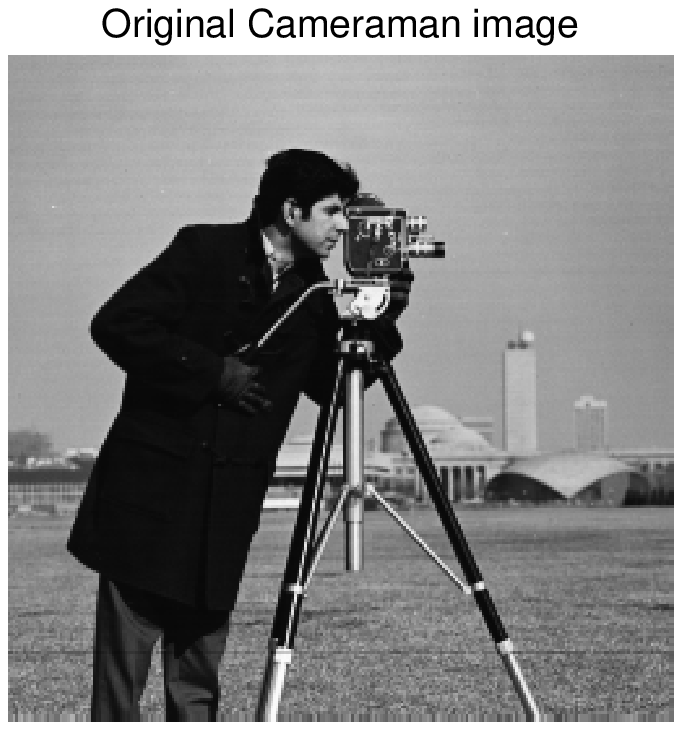}
  \end{minipage}
  \begin{minipage}[t]{0.45\linewidth}
  \centering
  \includegraphics[width=1.1\textwidth]{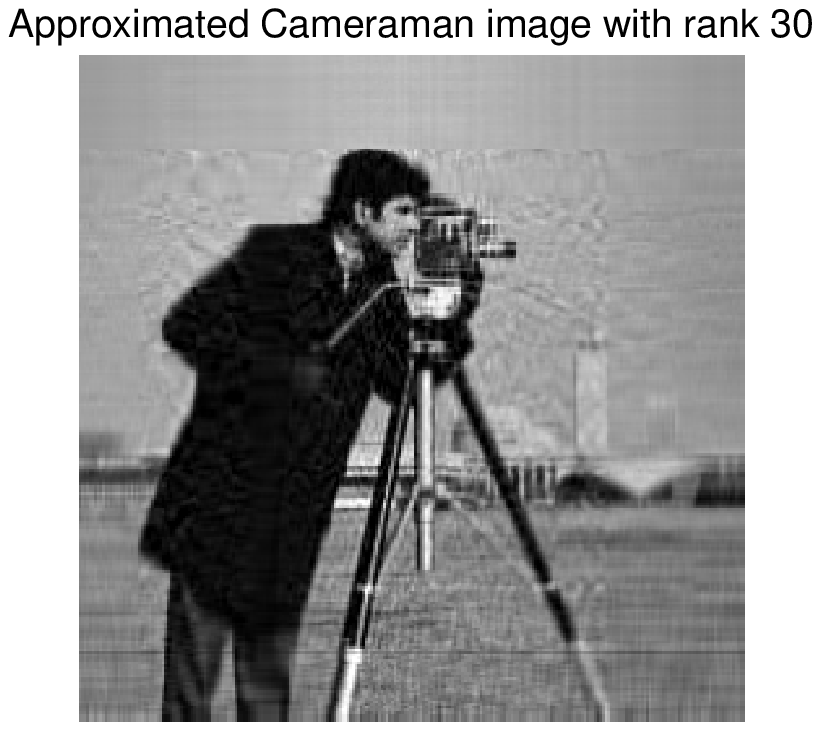}
  \end{minipage}
  \caption{Original $256\times 256$ Cameraman image and its approximated image with rank 30} \label{fig2}
\end{figure}

\begin{table*}\footnotesize
\caption{\scriptsize Numerical results of TLIHT, SVT and SVP algorithms for image inpainting problems}\label{table3}
\centering
\setlength{\tabcolsep}{3.5mm}{
\begin{tabular}{|c||l|l|l|l|l|l|l|l|}\hline
\multicolumn{9}{|c|}{$sr=0.40$}\\\hline
Image&\multicolumn{2}{c}{TLIHT, $\alpha=0.1$}&\multicolumn{2}{|c}{TLIHT, $\alpha=0.5$}&\multicolumn{2}{|c}{SVT}&\multicolumn{2}{|c|}{SVP}\\
\hline
(Name,\,rank,\,$fr$)&RE&Time&RE&Time&RE&Time&RE&Time\\
\hline
(Peppers, 30, 1.8129)& 3.66e-07& 5.08& 4.70e-07& 6.12&4.38e-02 &12.05 & 7.60e-01 & 1.79\\
\hline
(Cameraman, 30, 1.8129)& 1.08e-06& 11.85& 1.12e-06& 12.52& 7.99e-02& 7.84& 7.59e-01 & 2.35\\
\hline
\multicolumn{9}{|c|}{$sr=0.30$}\\\hline
Image&\multicolumn{2}{c}{TLIHT, $\alpha=0.1$}&\multicolumn{2}{|c}{TLIHT, $\alpha=0.5$}&\multicolumn{2}{|c}{SVT}&\multicolumn{2}{|c|}{SVP}\\
\hline
(Peppers, 30, 1.3597)& 3.36e-06& 33.16& 4.09e-06& 43.19&1.08e-01 &8.22& 8.25e-01 & 1.42\\
\hline
(Cameraman, 30, 1.3597)& 1.35e-05& 111.65& 2.97e-05& 173.958& 1.26e-01& 7.92& 8.26e-01 & 1.80\\
\hline
\end{tabular}}
\end{table*}

\begin{figure}[h!]
  \centering
  \begin{minipage}[t]{0.42\linewidth}
  \centering
  \includegraphics[width=1.1\textwidth]{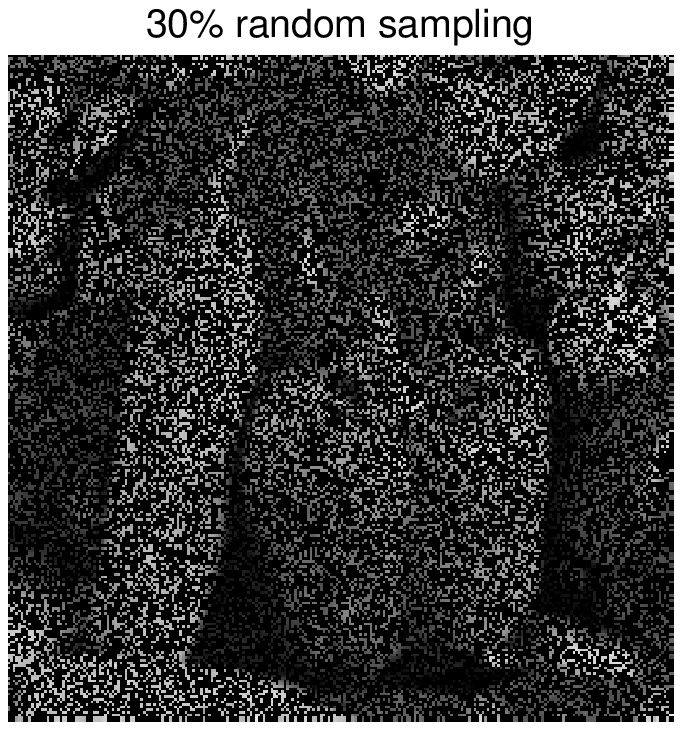}
  \end{minipage}\\
  \begin{minipage}[t]{0.42\linewidth}
  \centering
  \includegraphics[width=1.1\textwidth]{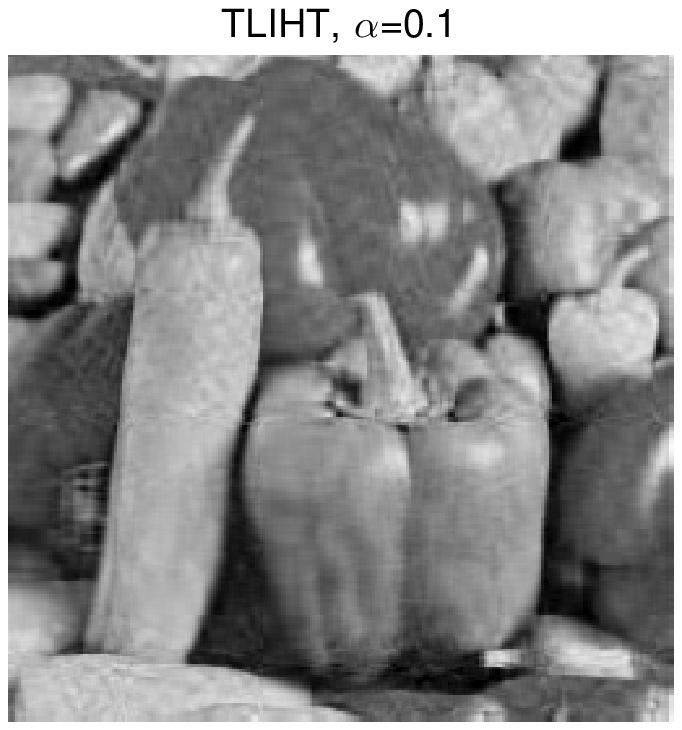}
  \end{minipage}
  \begin{minipage}[t]{0.42\linewidth}
  \centering
  \includegraphics[width=1.1\textwidth]{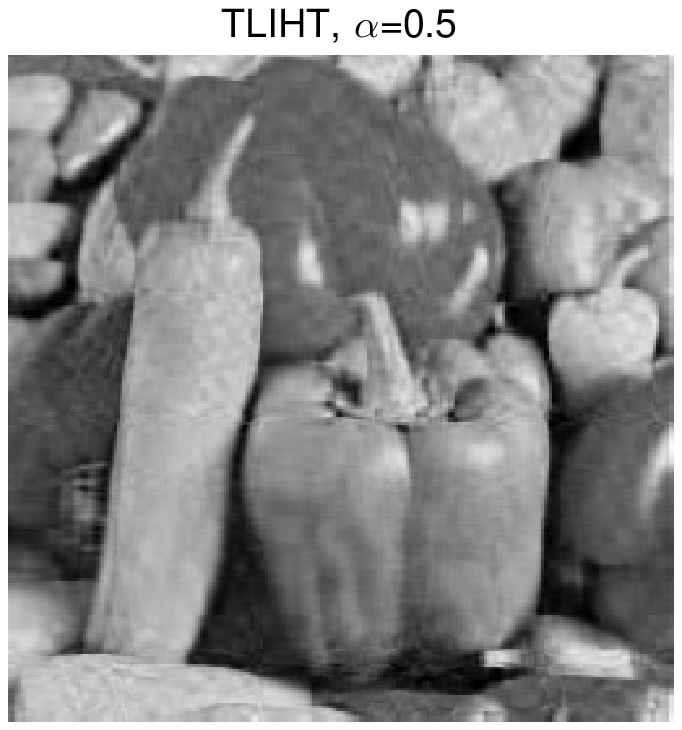}
  \end{minipage}
  \begin{minipage}[t]{0.42\linewidth}
  \centering
  \includegraphics[width=1.1\textwidth]{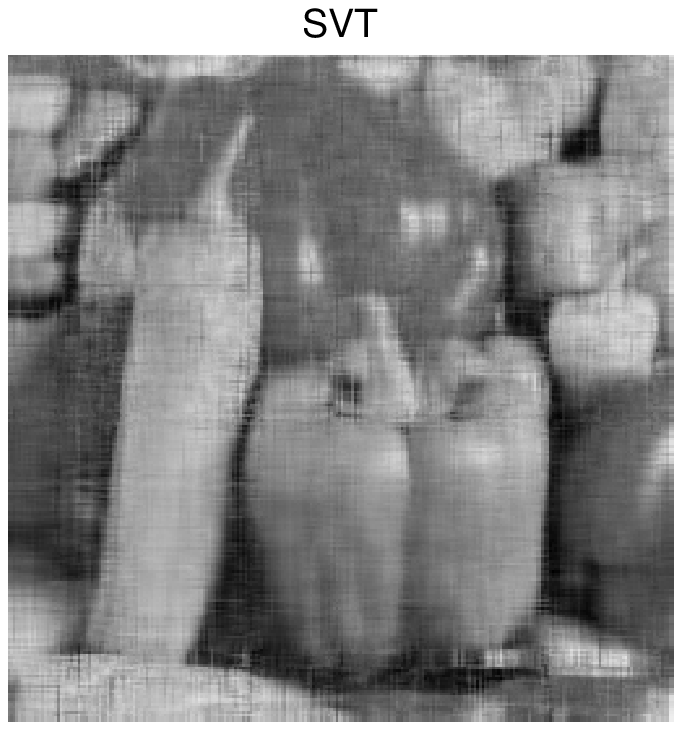}
  \end{minipage}
  \begin{minipage}[t]{0.42\linewidth}
  \centering
  \includegraphics[width=1.1\textwidth]{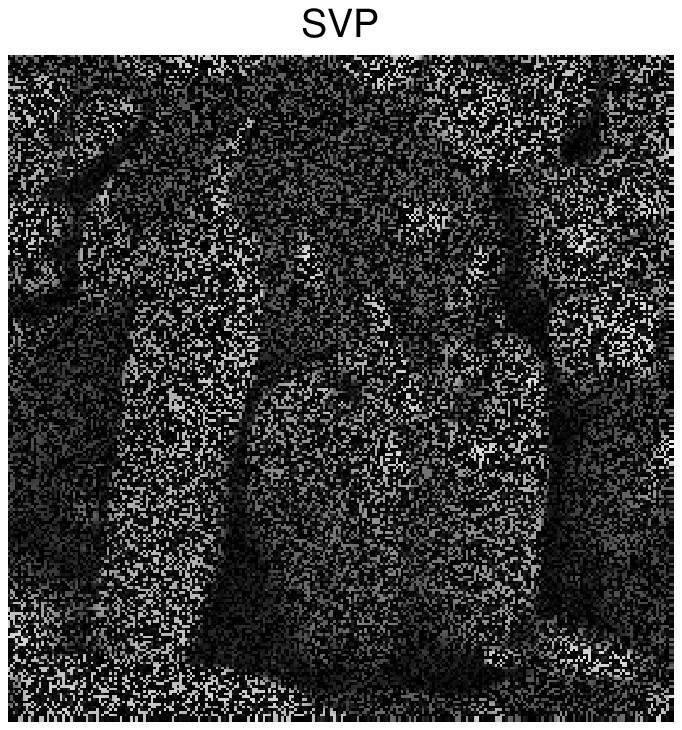}
  \end{minipage}
  \caption{Comparisons of TLIHT, SVT and SVP algorithms for recovering the approximated low-rank Peppers image with $sr=0.30$.} \label{fig3}
\end{figure}

\begin{figure}[h!]
  \centering
  \begin{minipage}[t]{0.42\linewidth}
  \centering
  \includegraphics[width=1.1\textwidth]{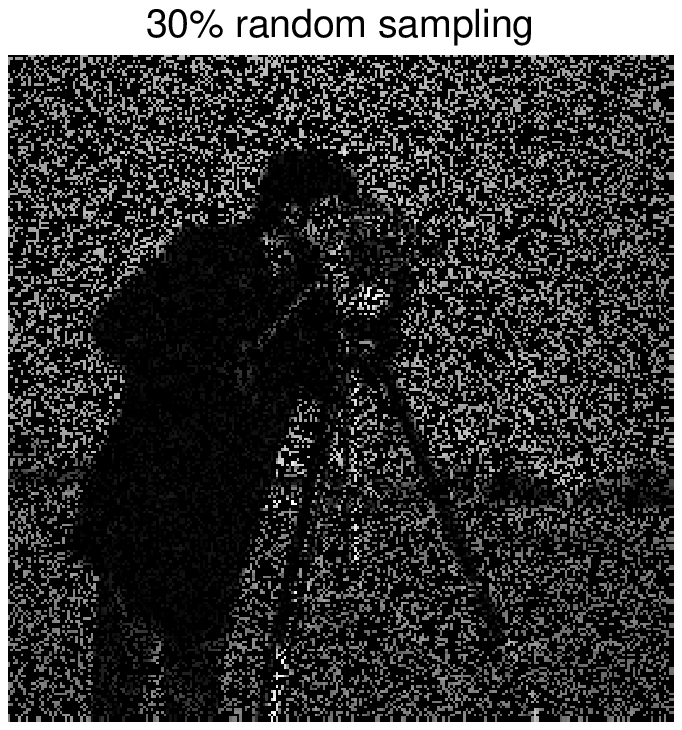}
  \end{minipage}\\
  \begin{minipage}[t]{0.42\linewidth}
  \centering
  \includegraphics[width=1.1\textwidth]{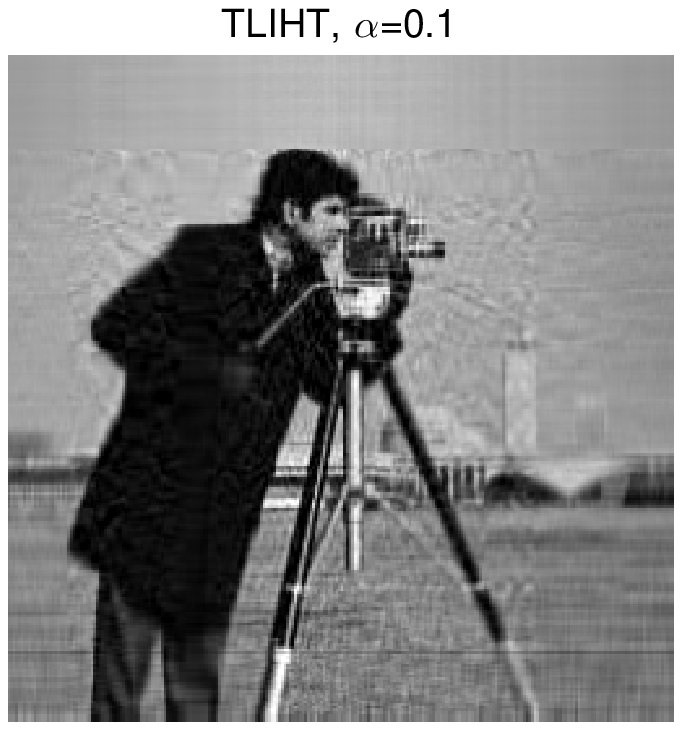}
  \end{minipage}
  \begin{minipage}[t]{0.42\linewidth}
  \centering
  \includegraphics[width=1.1\textwidth]{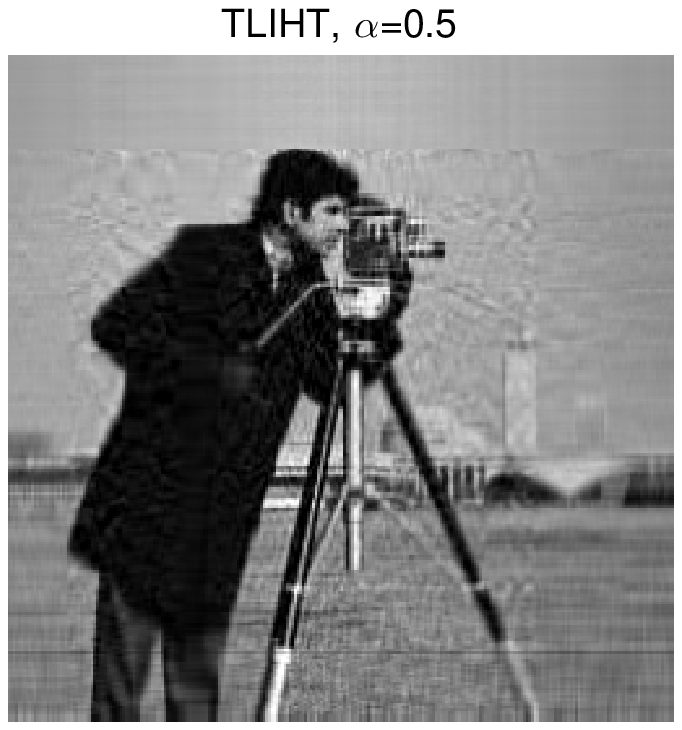}
  \end{minipage}
  \begin{minipage}[t]{0.42\linewidth}
  \centering
  \includegraphics[width=1.1\textwidth]{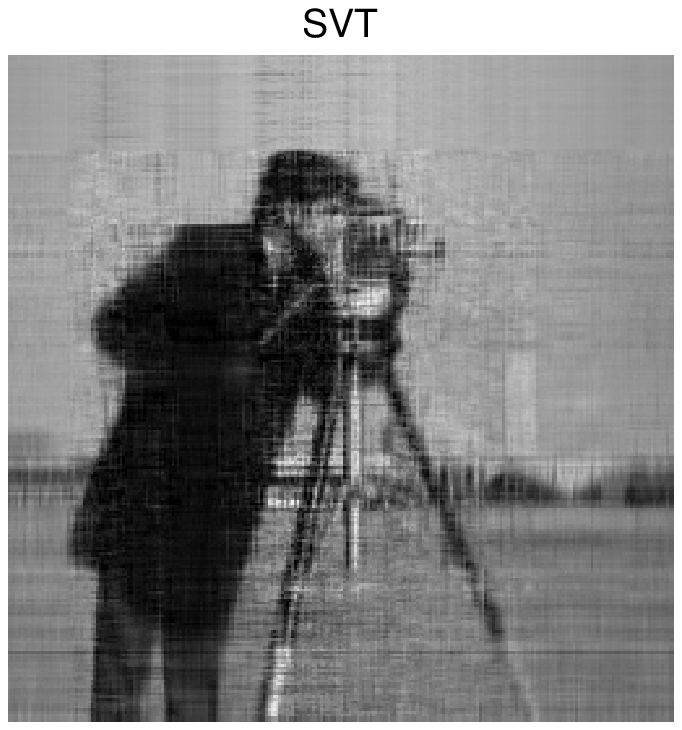}
  \end{minipage}
  \begin{minipage}[t]{0.42\linewidth}
  \centering
  \includegraphics[width=1.1\textwidth]{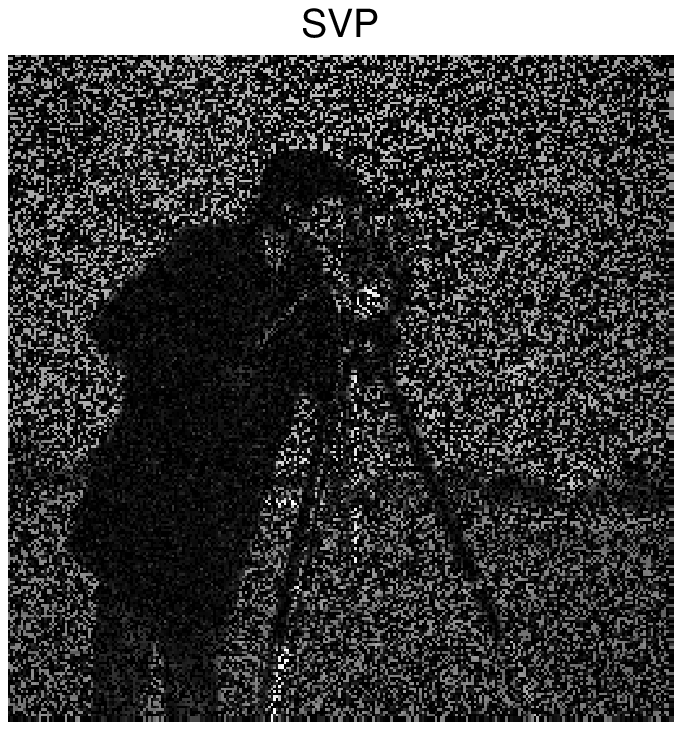}
  \end{minipage}
  \caption{Comparisons of TLIHT, SVT and SVP algorithms for recovering the approximated low-rank Cameraman image with $sr=0.30$.} \label{fig4}
\end{figure}

\section{Conclusion}\label{section6}

In this paper, we proposed a nonconvex function to approximate the rank function in the NP-hard problem (AMRM), and studied the transformed minimization problem in terms of theory,
algorithm and computation. We discussed the equivalence of problem (AMRM) and (TLAMRM), and the uniqueness of global minimizer of the problem (TLAMRM) also solves the NP-hard problem
(AMRM) if the linear map $\mathcal{A}$ satisfies a restricted isometry property (RIP). In addition, an iterative thresholding algorithm is proposed to solve the regularization problem
(RTLAMRM). Numerical results on low-rank matrix completion problems illustrated that our algorithm is able to recover a low-rank matrix, and the extensive numerical on image
inpainting problems shown that our algorithm performs the best in finding a low-rank image compared with some state-of-art methods.

\section*{Acknowledgment}
The work was supported by the National Natural Science Foundations of China (11771347, 91730306, 41390454, 11271297) and the Science Foundations of Shaanxi Province of
China (2016JQ1029, 2015JM1012).

\ifCLASSOPTIONcaptionsoff
  \newpage
\fi

\end{document}